\documentclass[11pt]{amsart}

\usepackage{amsmath}
\usepackage{amsfonts}
\usepackage{amssymb}
\usepackage{amsthm}

\usepackage{graphicx} 
\usepackage{verbatim} 

\usepackage{color}

\newcommand{\mcm}[3]{\newcommand{#1}[#2]{{\ensuremath{#3}}}} 
\mcm{\NN}{0}{\mathbb{N}}
\mcm{\Tcal}{0}{\mathcal{T}}
\mcm{\Scal}{0}{\mathcal{S}}
\mcm{\Pcal}{0}{\mathcal{P}}
\mcm{\Bcal}{0}{\mathcal{B}}
\mcm{\Fcal}{0}{\mathcal{F}}
\mcm{\Ucal}{0}{\mathcal{U}}
\mcm{\abs}{1}{\left\lvert #1 \right\rvert}
\mcm{\defiff}{0}{:\Leftrightarrow}
\mcm{\defeq}{0}{:=}
\mcm{\eqdef}{0}{=:}
\mcm{\sm}{0}{\setminus}
\mcm{\se}{0}{\subseteq}
\mcm{\ct}{0}{^\complement}
\mcm{\restricted}{0}{{\upharpoonright}}

\newcommand{\ignore}[1]{}

\DeclareMathOperator{\interior}{int}
\DeclareMathOperator{\corner}{cor}

\DeclareMathOperator{\Aut}{Aut}

\newtheorem{theorem}{Theorem}[section]
\newtheorem{thm_intro}{Theorem}
\theoremstyle{plain}

\numberwithin{subcase}{case}

\newtheorem{corollary}[theorem]{Corollary}
\newtheorem{lemma}[theorem]{Lemma}

\theoremstyle{definition}
\newtheorem{remark}[theorem]{Remark}
\newtheorem{construction}[theorem]{Construction}
\newtheorem{example}[theorem]{Example}
\numberwithin{equation}{section}

\title[]{Canonical tree-decompositions of a graph that display its $k$-blocks}

\author[]{Johannes Carmesin and J. Pascal Gollin}

\begin{document}

\begin{abstract}
    A $k$-\emph{block} in a graph $G$ is a maximal set of at least $k$ vertices no two of which can be separated in $G$ by removing less than $k$ vertices.
    It is \emph{separable} if there exists a tree-decomposition of adhesion less than $k$ of $G$ in which this $k$-block appears as a part.
    
    Carmesin, Diestel, Hamann, Hundertmark and Stein proved that every  finite graph has a canonical tree-decomposition of adhesion less than $k$ that distinguishes all its $k$-blocks and tangles of order $k$.
    We construct such tree-decompositions with the additional property that every separable $k$-block is equal to the unique part in which it is contained.
    This proves a conjecture of Diestel.
\end{abstract}

\maketitle

\section{Introduction}\label{section:intro}

\emph{Tangles} in a graph $G$ are orientations of the low order separations that consistently point towards some `highly connected piece' of $G$.
As a fundamental tool for their graph minors project, Robertson and Seymour~\cite{RS_GM10} proved that every finite graph has a tree-decomposition that distinguishes every two maximal tangles.

More recently, $k$-\emph{profiles} were introduced as a common generalisation of $k$-tangles and $k$-blocks~\cite{profiles}.
Here, a $k$-\emph{block} in a graph $G$ is a maximal set of at least $k$ vertices no two of which can be separated in $G$ by removing less than $k$ vertices.
Carmesin, Diestel, Hamann and Hundertmark showed that every graph has a canonical tree-decomposition of adhesion less than $k$ that distinguishes all its $k$-profiles~\cite{canon1}. 

In \cite{canon2}, these authors asked how one could improve the above tree-decompo{-}sitions further 
so that they also display the structure of the $k$-blocks:
it would be nice if we could compress any part containing a $k$-block so that it does not contain any `junk'.

In this paper, we prove that this is possible simultaneously for all $k$-blocks that can be isolated at all in a tree-decomposition, canonical or not.
More precisely, we call a $k$-block \emph{separable} if it appears as a part in some tree-decomposition of adhesion less than $k$ of $G$.
We prove the following, which was conjectured by Diestel~\cite{D14:SpiekeroogProblems} (see also~\cite{canon2}).

\begin{thm_intro}\label{Intro_Diestel_conj}
    Every finite graph $G$ has a canonical tree-decomposition~$\Tcal$ \linebreak of adhesion less than $k$ that distinguishes efficiently every two distinct \linebreak $k$-profiles, 
    and which has the further property that
    every separable $k$-block is equal to the unique part of~$\Tcal$ in which it is contained.
\end{thm_intro}

We also prove the following related result:

\begin{thm_intro}\label{main_thm_intro}
    Every finite graph $G$ has a canonical tree-decomposition~$\Tcal$ that distinguishes efficiently every two distinct maximal robust profiles, 
    and which has the further property that
    every separable block inducing a maximal robust profile is equal to the unique part of~$\Tcal$ in which it is contained.
\end{thm_intro}

See Section \ref{section:prelims} for a definition of robust and \cite{tree_structure} for an example showing that Theorem~\ref{main_thm_intro} fails if we leave out `robust'.
Theorem~\ref{main_thm_intro} without its description of the separable blocks is a result of Hundertmark and Lemanczyk~\cite{profiles}, which implies the aforementioned theorem of Robertson and Seymour.
In Section~\ref{section:proof}, we give an example showing that it is impossible to ensure that non-maximal robust separable blocks are also displayed by a tree-decomposition which distinguishes all the maximal robust profiles efficiently.

After recalling some preliminaries in Section~\ref{section:prelims}, we develop the necessary tools in Section~\ref{section:construction}.
Then we prove our main result in Section~\ref{section:proof}.

\section{Preliminaries}\label{section:prelims}

Unless otherwise mentioned, $G$ will always denote a finite, simple and undirected graph with vertex set $V(G)$ and edge set $E(G)$.
Any graph-theoretic term and notation not defined here are explained in \cite{diestel_graph_theory}.

A vertex is called \emph{central} in $G$ if the greatest distance to any other vertex is minimal.
It is well known that a finite tree $T$ has either a unique central vertex or precisely two central adjacent vertices $v$ and $w$. In the second case $vw$ is called a \emph{central edge}.
For a vertex or edge to be central is obviously a property invariant under automorphisms of $G$.

Let us recall some notations from \cite{canon1}. 

\subsection{Separations}
An ordered pair $(A,B)$ of subsets of $V(G)$ is a \emph{separation} of $G$ if ${A \cup B = V(G)}$ and if there is no edge $e = vw \in E(G)$ with $v \in A \sm B$ and $w \in B \sm A$.
The cardinality $\abs{A \cap B}$ of the \emph{separator} $A \cap B$ of a separation $(A,B)$ is the \emph{order} of $(A,B)$ and a separation of order $k$ is a $k$-\emph{separation}.

A separation $(A,B)$ is \emph{proper} if neither $A \subseteq B$ nor $B \subseteq A$.
Otherwise $(A,B)$ is \emph{improper}.
A separation $(A,B)$ is \emph{tight} if every vertex in $A \cap B$ has a neighbour in $A \sm B$ and a neighbour in $B \sm A$.

The set of separations of $G$ is partially ordered via
\[
    (A,B) \leq (C,D)\ \defiff\ A \subseteq C\ \wedge\ D \subseteq B.
\]

For no two proper separations $(A,B)$ and $(C,D)$, the separation $(A,B)$ is {$\leq$-comparable} with $(C,D)$ and $(D,C)$. In particular we obtain that $(A,B)$ and $(B,A)$ are not {$\leq$-comparable}.

A separation $(A,B)$ is \emph{nested} with a separation $(C,D)$ if $(A,B)$ is \linebreak {$\leq$-comparable} with either $(C,D)$ or $(D,C)$. Since
\[
    (A,B) \leq (C,D)\ \iff\ (D,C) \leq (B,A),
\]
being nested is symmetric and reflexive. Separations that are not nested are called \emph{crossing}.

A separation $(A,B)$ is \emph{nested} with a set $S$ of separations if $(A,B)$ is nested with every $(C,D) \in S$.
A set $S$ of separations is \emph{nested} with a set $S'$ of separations if every $(A,B) \in S$ is nested with $S'$ or equivalently every $(C,D) \in S'$ is nested with $S$.

A set $N$ of separations is \emph{nested} if its elements are pairwise nested.
A set $S$ of separations is \emph{symmetric} if for every $(A,B) \in S$ it also contains its \emph{inverse} separation $(B,A)$.
A symmetric set $S$ of separations is also called a \emph{separation system} or a \emph{system of separations}, and if all its separations are proper, $S$ is called a \emph{proper separation system}.
For a set $S$ of separations the separation system \emph{generated by $S$} is the separation system consisting of the separations in $S$ and their inverses.
A set $S$ of separations is \emph{canonical} if it is invariant under the automorphisms of $G$, i.e.\ for every $(A,B) \in S$ and for every $\varphi \in \Aut(G)$ we obtain $(\varphi[A],\varphi[B]) \in S$.

A separation $(A,B)$ \emph{separates} a vertex set $X \subseteq V(G)$ if $X$ meets both $A \sm B$ and $B \sm A$.
Given a set $S$ of separations a vertex set $X \subseteq V(G)$ is $S$-\emph{inseparable} if no separation $(A,B) \in S$ separates $X$.
A maximal $S$-inseparable vertex set is an $S$-\emph{block} of $G$. 

For $k \in \NN$ let $S_{< k}$ denote the set of separations of order less than $k$ of~$G$.
The $(< k)$-\emph{inseparable} sets are the $S_{< k}$-inseparable sets.
So the $k$-\emph{blocks} are exactly the $S_{< k}$-blocks of size at least $k$.

For two separations $(A,B)$ and $(C,D)$ not equal to $(V(G),V(G))$ consider a \emph{cross-diagram} as in Figure~\ref{fig:cross_diagram}.
Every pair ${(X,Y) \in \{A,B\} \times \{C,D\}}$ denotes a \emph{corner} of this cross-diagram, which we also denote by $\corner(X,Y)$.
Let ${\overline{X} \in \{A,B\} \sm \{X\}}$ and ${\overline{Y} \in \{C,D\} \sm \{Y\}}$.
In the diagram we consider the \emph{center} ${c \defeq  A \cap B \cap C \cap D}$ and for a corner $\corner(X,Y)$ as above the \emph{interior} ${\interior(X,Y) \defeq (X \cap Y) \sm (\overline{X} \cup \overline{Y})}$ and the \emph{links} ${\ell_X \defeq (X \cap Y \cap \overline{Y}) \sm c}$ and ${\ell_Y \defeq (Y \cap X \cap \overline{X}) \sm c}$.
The vertex set ${X \cap Y}$ is the disjoint union of ${\interior(X,Y)}$ with $\ell_X$, $\ell_Y$ and $c$ and thus can be associated with the corner $\corner(X,Y)$.

\begin{center}
    \begin{figure}[htpb]
        \begin{center}
            \includegraphics[height=4.5cm]{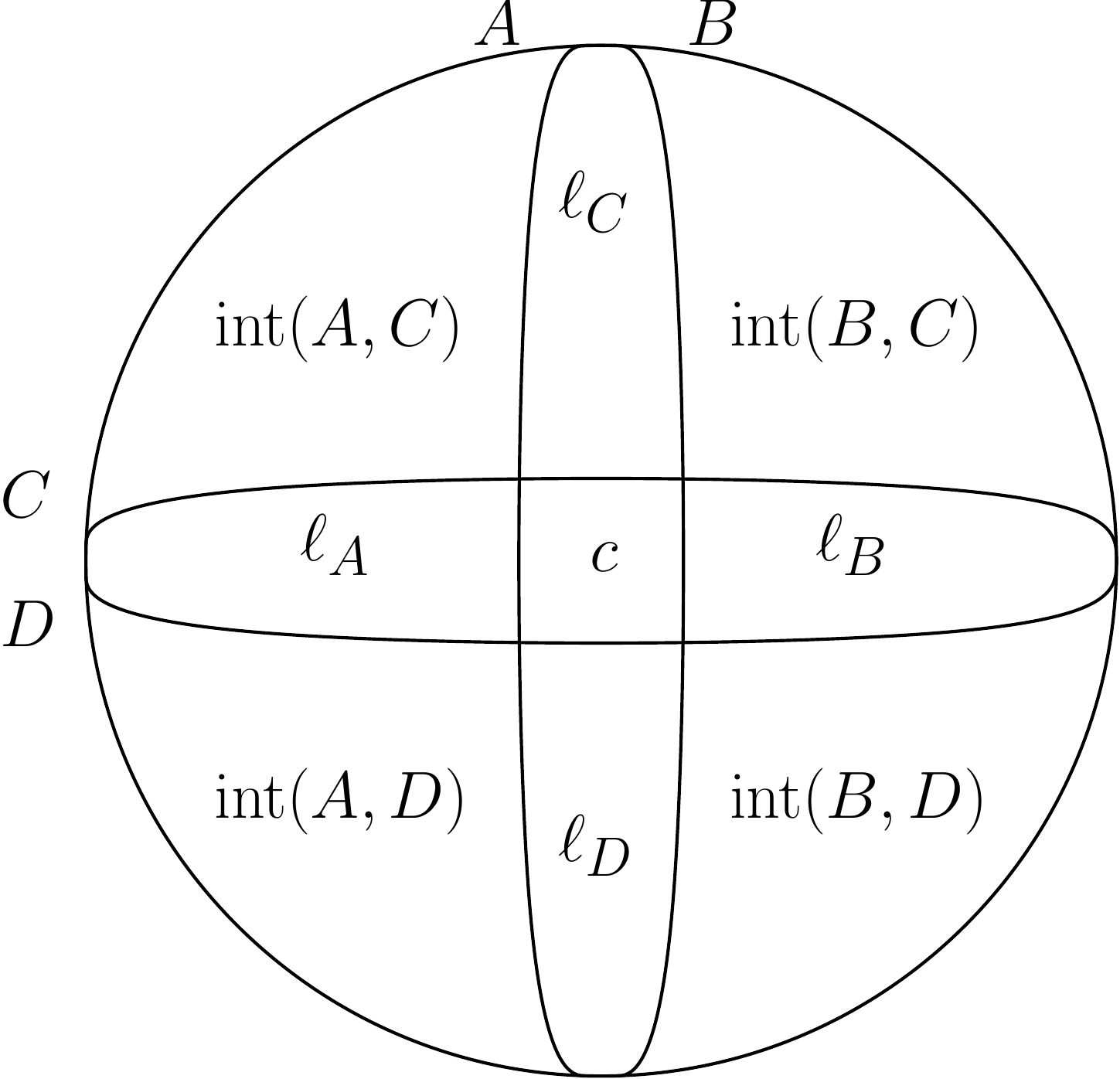} 
            \caption{cross-diagram for $(A,B)$ and $(C,D)$}
            \label{fig:cross_diagram}
        \end{center}
    \end{figure}
\end{center}

\begin{remark}\label{corner_lemma}
    Two separations $(A,B)$ and $(C,D)$ are nested, if and only if for one of their corners $\corner(X,Y)$ the interior $\interior(X,Y)$ and its links $\ell_X$ and $\ell_Y$ are empty.
    \qed
\end{remark}

For a corner $\corner(X,Y)$ there is a \emph{corner separation} ${(X \cap Y, \overline{X} \cup \overline{Y})}$, which is again a separation of $G$.

\begin{lemma}\label{fish_lemma}
    \emph{\cite[Lemma 2.2]{tree_structure}}
    For two crossing separations $(A,B)$ and $(C,D)$ any of its corner separation is nested with every separation that is nested with both $(A,B)$ and $(C,D)$.
\end{lemma}

In particular a corner separation is nested with $(A,B)$, $(C,D)$ and all corner separations.
A double counting argument yields:

\begin{remark}\label{submodularity}
    For any two separations $(A,B)$ and $(C,D)$, the orders of the separations $(A \cap C, B \cup D)$ and $(B \cap D, A \cup C)$ sum to $\lvert A \cap B \rvert + \lvert C \cap D \rvert$.
    \qed
\end{remark}

\subsection{Tree-decompositions}
Recall that a \emph{tree-decomposition} $\Tcal$ of $G$ is a pair $\big(T,(P_t)_{t \in V(T)}\big)$ of a tree $T$ and a family of vertex sets $P_t \subseteq V(G)$ for every node $t \in V(T)$, such that
\begin{enumerate}
    \item[(T1)] $V(G) = \bigcup_{t \in V(T)} P_t$;
    \item[(T2)] for every edge $e \in E(G)$ there is a node $t \in V(T)$ such that both end vertices of $e$ lie in $P_t$;
    \item[(T3)] whenever $t_2$ lies on the $t_1$ -- $t_3$ path in $T$ we obtain $P_{t_1} \cap P_{t_3} \subseteq P_{t_2}$.
\end{enumerate}

The sets $P_t$ are the \emph{parts} of $\Tcal$.
For an edge $tt' \in E(T)$ the intersection $P_{t} \cap P_{t'}$ is the corresponding \emph{adhesion set} and the maximum size of an adhesion set of $\Tcal$ is the \emph{adhesion} of $\Tcal$.
A node $t \in V(T)$ is a \emph{hub node} if the corresponding part $P_t$ is a subset of $P_{t'}$ for some neighbour $t'$ of $t$. If $t$ is a hub node, then $P_t$ is a \emph{hub}.
A tree-decomposition $\Tcal = \big(T, (P_t)_{t \in  V(T)}\big)$ of $G$ and a tree-decomposition $\Tcal' = \big(T', (P'_t)_{t \in  V(T')}\big)$ of $G'$ are \emph{isomorphic} if there is an isomorphism $\varphi: G \to G'$ and an isomorphism $\psi: T \to T'$ such that for every part $P_t$ of $\Tcal$ we obtain $\varphi[P_t] = P'_{\psi(t)}$.
We say $\varphi$ \emph{induces} an isomorphism between $\Tcal$ and $\Tcal'$. 
A tree-decomposition $\Tcal$ is \emph{canonical} if it is invariant under the automorphisms of $G$, i.e.\ every automorphism of $G$ induces an automorphism of $\Tcal$.

Let $\big(T,(P_t)_{t \in V(T)}\big)$ be a tree-decomposition of $G$.
For $t \in V(T)$ the \emph{torso} $H_t$ is the graph obtained from $G[P_t]$ by adding all edges joining two vertices in a common adhesion set $P_{t} \cap P_{u}$ for any $tu \in E(T)$.
A separation $(A,B)$ of~$G[P_t]$ is a separation of $H_t$ if and only if it does not separate any adhesion set $P_t \cap P_{t'}$ for $tt' \in E(T)$.
A separation $(A,B)$ of $G$ with $A \cap B \subseteq P_t$ for some node $t \in V(T)$ that does not separate any adhesion set $P_t \cap P_{t'}$ for $tt' \in E(T)$ \emph{induces} the separation $(A \cap P_t, B \cap P_t)$ of $H_t$.

Every oriented edge $\vec{e} = t_1t_2$ of $T$ divides $T-e$ in two components $T_1$ and $T_2$ with $t_1 \in V(T_1)$ and $t_2 \in V(T_2)$.
By \cite[Lemma 12.3.1]{diestel_graph_theory} $\vec{e}$ \emph{induces} the separation $\big( \bigcup_{t \in V(T_1)} P_t, \bigcup_{t \in V(T_2)} P_t \big)$ of $G$ such that the separator coincides with the adhesion set $P_{t_1} \cap P_{t_2}$.
We say a separation is \emph{induced} by $\Tcal$ if it is induced by an oriented edge of $T$. 

The set of separations induced by a tree-decomposition $\Tcal$ (of adhesion less than $k$) is a nested system $N(\Tcal)$ of separations (of order less than $k$).
We say $N(\Tcal)$ is \emph{induced} by $\Tcal$. Clearly if $\Tcal$ is canonical, then so is $N(\Tcal)$.

Conversely, as proven in \cite{tree_structure}, every nested separation system $N$ \emph{induces} a tree-decomposition $\Tcal(N)$:

\begin{theorem}\label{tree_from_N}
    \emph{\cite[Theorem 4.8]{tree_structure}}
    Let $N$ be a canonical nested separation system of $G$. Then there is a canonical\,\footnote{In the original paper this theorem is stated without the canonicity since it holds in a greater generality. But it is clear from the proof that if $N$ is canonical, then so is $\Tcal(N)$.} tree-decomposition $\Tcal(N)$ of $G$ such that
    \begin{enumerate}
        \item[(i)] every $N$-block of $G$ is a part of $\Tcal(N)$;
        \item[(ii)] every part of $\Tcal(N)$ is either an $N$-block of $G$ or a hub;
        \item[(iii)] the separations of $G$ induced by $\Tcal(N)$ are precisely those in $N$;
        \item[(iv)] every separation in $N$ is induced by a unique oriented edge of $\Tcal(N)$.
    \end{enumerate}
\end{theorem}

\subsection{Profiles}
Let $S$ be a separation system.
A subset $O \subseteq S$ is an \emph{orientation} of~$S$ if for every $(A,B) \in S$ exactly one of $(A,B)$ and $(B,A)$ is an element of~$O$.
An orientation $O$ of $S$ is \emph{consistent} if for every $(A,B)$, $(C,D) \in S$ with $(A,B) \in O$ and $(C,D) \leq (A,B)$ we obtain $(C,D) \in O$ as well. 
A consistent orientation $P$ of $S_{< k}$ is called a $k$-\emph{profile} 
if it satisfies
\begin{enumerate}
    \item[(P)] for all $(A,B)$, $(C,D) \in P$ we have $(B \cap D, A \cup C) \notin P$.
\end{enumerate}
In particular if the order $\abs{(A \cup C) \cap (B \cap D)}$ of this corner separation is less than $k$, we have $(A \cup C, B \cap D) \in P$.
Sometimes we omit the $k$ and call $P$ a \emph{profile}.

It is easy to check that every $k$-block $b$ \emph{induces} a $k$-profile via
\[
    P_k(b) \defeq \big\{ (A,B) \in S_{< k}\ \big|\ b \subseteq B \big\}.
\]
Also \emph{tangles} of order $k$ (or $k$-\emph{tangles}), as introduced by Robertson and Seymour \cite{RS_GM10}, are $k$-profiles.
For more background on profiles, see \cite{profiles}. 

For $r \in \NN$, a $k$-profile $P$ is $r$-\emph{robust} if for any ${(A,B) \in P}$ and any ${(C,D) \in S_{< r+1}}$ one of $(A \cup C, B \cap D)$, $(A \cup D, B \cap C)$ either has order at least $k-1$, or is in $P$.
If $P$ is $r$-robust for all $r \in \NN$, then we call $P$ \emph{robust}.

A robust $k$-profile $P$ is \emph{maximal} if there does not exist a robust $\ell$-profile~$Q$ with $P \subsetneq Q$ and $\ell > k$.
Then $P$ is just called a \emph{maximal robust profile}.

\begin{remark} \label{robust}
    \begin{itemize}
        \item[(i)] Every $k$-profile is $\ell$-robust for all $\ell < k$;
        \item[(ii)] if a $k$-block $b$ contains a complete graph on $k$ vertices, then the induced $k$-profile $P_{k}(b)$ is robust. \qed 
    \end{itemize}
\end{remark}


The next lemma basically states that every $k$-profile induces a $k$-haven, as introduced by Seymour and Thomas \cite{havens}.

\begin{lemma}\label{profiles_are_havens}
    Let ${X \subseteq V(G)}$ with ${\abs{X} < k}$ and let $Q$ be a {$k$-profile}.
    Then there exists a component $C$ of ${G-X}$ such that ${(V(G) \sm C, C \cup X) \in Q}$.
    
    Furthermore, ${(V(G) \sm C, C \cup N(C)) \in Q}$ as well.
\end{lemma}

\begin{proof}
    Let $C_1, \ldots, C_n$ denote the components of $G - X$ and for ${i \in \{1, \ldots, n\}}$ let ${(A_i,B_i) \defeq (V(G) \sm C_i, C_i \cup X)}$. 
    To reach a contradiction suppose that $(B_i,A_i) \in Q$ for all ${i \in \{1, \ldots, n\}}$.
    Then (P) yields inductively for all $m \leq n$ that ${\big(\bigcup_{i \leq m} B_i, \bigcap_{i \leq m} A_i \big) \in Q}$, since their separators all equal $X$.
    Hence for ${m = n}$, we obtain ${(V(G),X) \in Q}$, contradicting the consistency of $Q$ with ${(X,V(G)) \leq (V(G),X)}$. Thus there is a component $C$ of $G-X$ such that ${(A,B) \defeq (V(G) \sm C, C \cup X) \in Q}$.
    
    Now suppose $(C \cup N(C), V(G) \sm C) \in Q$. Then (P) with $(A,B)$ yields that
        ${\big((V(G)\sm C) \cup C \cup N(C), (C \cup X) \cap (V(G) \sm C)\big) = (V(G),X) \in Q},$
    contradicting the consistency of $Q$ again.
\end{proof}

A $k$-profile $Q$ \emph{inhabits} a part $P_t$ of a tree-decomposition $\big(T,(P_t)_{t \in V(T)}\big)$ if for every $(A,B) \in Q$ we obtain that $(B \sm A)  \cap P_t$ is not empty. 
Note that if for a node $t$ every separation induced by an oriented edge $ut$ of $T$ has order less than $k$, then $Q$ inhabits $P_t$ if and only if all those separations are in $Q$.

\begin{corollary}
    \label{profiles_inhabit_large_parts}
    Let $\big( T, (P_t)_{t \in V(T)} \big)$ be a tree-decomposition and let $Q$ be a $k$-profile. If $Q$ inhabits a part $P_t$, then $\abs{P_t} \geq k$.
\end{corollary}

\begin{proof}
    Our aim is to show that if $\abs{P_t} < k$, then any $k$-profile $Q$ does not inhabit $P_t$.
    By Lemma~\ref{profiles_are_havens} there is a component $C$ of $G - P_t$ such that $(V(G) \sm C, C \cup P_t) \in Q$.
    Since $(C \cup P_t) \sm (V(G) \sm C) = C$ and since $C \cap P_t$ is empty, we obtain that $Q$ does not inhabit $P_t$.
\end{proof}

A set $\Pcal$ of profiles is \emph{canonical} if for every $P \in \Pcal$ and every automorphism~$\varphi$ of $G$ the profile ${\big\{ \big(\varphi[A],\varphi[B]\big)\ \big|\ (A,B) \in P \big\}}$ is also in $\Pcal$.

Two profiles $P$ and $Q$ are \emph{distinguishable} if there is a separation $(A,B)$ with $(A,B) \in P$ and $(B,A) \in Q$.
Such a separation \emph{distinguishes} $P$ and~$Q$.
It is said to distinguish $P$ and $Q$ \emph{efficiently} if its order $\abs{A \cap B}$ is minimal among all separations distinguishing $P$ and $Q$.
A set $\Pcal$ of profiles is \emph{distinguishable} if every two distinct ${P, Q \in \Pcal}$ are distinguishable.
A tree-decomposition $\Tcal$ \emph{distinguishes} two profiles $P$ and $Q$ (efficiently) if some ${(A,B)}$ induced by $\Tcal$ distinguishes them (efficiently).

For our main result of this paper, we will build on the following theorem.

\begin{theorem}\label{canon_td_profiles} 
    \emph{\cite[Theorem~2.6]{profiles}}\footnote{Since \cite{profiles} is unpublished, see also \cite[Theorem 6.3]{tree_structure} for a version just concerning robust blocks or \cite[Theorem 9.2]{end_structure} for a version also dealing with infinite graphs.}
    Every graph $G$ has a canonical tree-decomposition of adhesion less than $k$ that distinguishes every two distinguishable $(k-1)$-robust $\ell$-profiles of $G$ for some values $\ell \leq k$ efficiently.
    
    Moreover, every separation induced by the tree-decomposition distinguishes some of those profiles efficiently.
\end{theorem}

\section{Construction methods}\label{section:construction}

\subsection{Sticking tree-decompositions together}

Given a tree-decomposi{-}tion $\Tcal$ of~$G$ and for each torso $H_t$ a tree-decomposition $\Tcal^t$, 
our aim is to  construct a new tree-decomposition~$\overline{\Tcal}$ of~$G$ by gluing together the tree-decompositions~$\Tcal^t$ of the torsos along $\Tcal$ in a canonical way.

\begin{example}\label{example:canon_glue}
    First we shall give the construction of~$\overline{\Tcal}$ for a particular example:
    $G$ is obtained from three edge-disjoint triangles intersecting in a single vertex by identifying two other vertices of distinct triangles.
    The tree-decomposition $\Tcal$ of~$G$ and the tree-decompositions of the torsos are depicted in Figure~\ref{fig:ex1}\,(a).
    In order to stick the tree-decompositions of the torsos together in a canonical way, we first have to refine them, see Figure~\ref{fig:ex1}\,(b).
    \begin{figure}[htpb]
    \includegraphics[height=3.5cm]{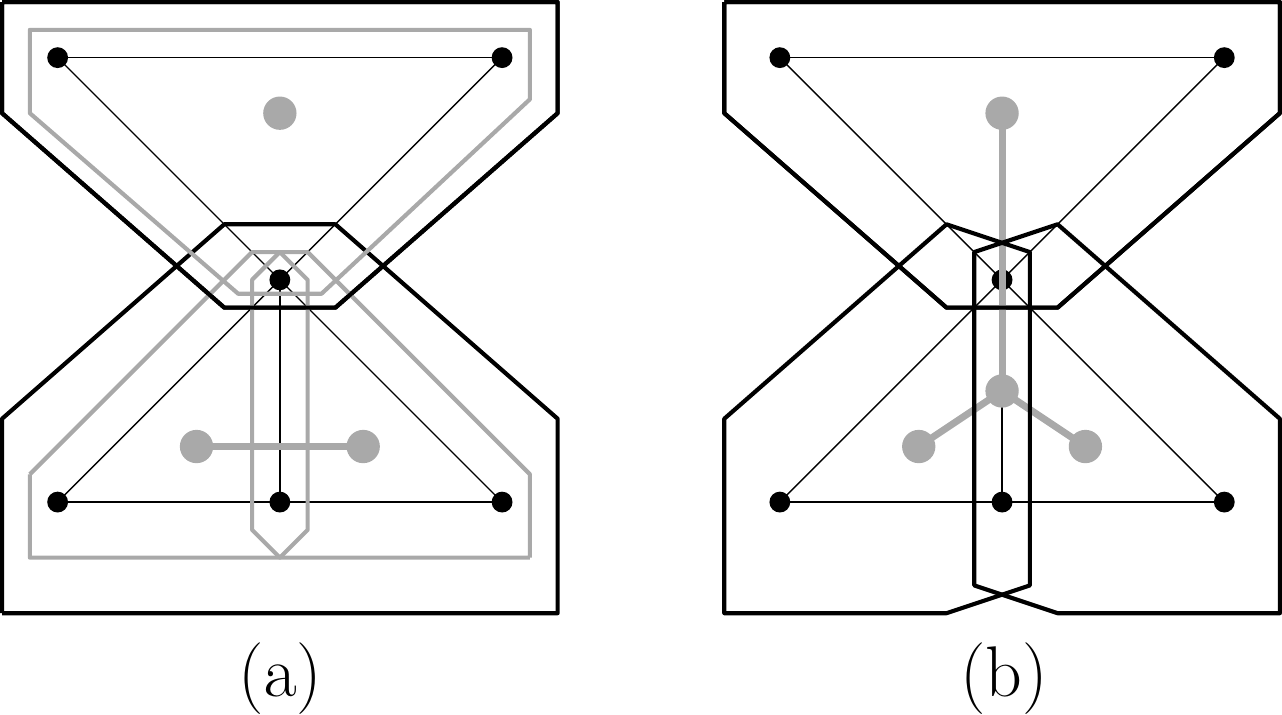} 
        \caption{(a) shows the tree-decomposition $\Tcal$ of~$G$, drawn in black, and the tree-decompositions of the torsos, drawn in grey.
            (b) shows the canonically glued tree-decomposition $\overline{\Tcal}$.}
        \label{fig:ex1}
    \end{figure}
\end{example}

Before we can construct $\overline \Tcal$, we need some preparation. 

\begin{construction}\label{constr:T_hat}
    Given a tree-decomposition $\Tcal = \big(T,(P_t)_{t \in V(T)}\big)$ of $G$, we construct a new tree-decomposition $\widetilde{\Tcal} = \big(\widetilde{T},({P}_t)_{t \in V(\widetilde{T})}\big)$ of~$G$ by contracting every edge $tu$ of $T$ where $P_t = P_u$.\footnote{Here we understand the nodes of $\widehat{T}$ to be nodes of $T$, where a node obtained through the contraction of an edge $tu$ to be identified with either $t$ or $u$.}
    In this tree-decomposition two adjacent nodes never have the same part.
    Let $F \subseteq E(\widetilde{T})$ be the set of edges $tu$ where neither ${P}_{t} \subseteq {P}_{u}$ nor ${P}_{u} \subseteq {P}_{t}$.
    By subdividing every edge $tu \in F$ and assigning to the subdivided node $x$ the part ${P}_{x} \defeq {P}_{t} \cap {P}_{u}$, we obtain a new tree-decomposition $\widehat{\Tcal} = \big(\widehat{T},({P}_t)_{t \in V(\widehat{T})}\big)$.
\end{construction}

\begin{remark}
    \label{gluing_aux_td}
    $\widehat{\Tcal}$ defined as in Construction~\ref{constr:T_hat} satisfies the following:
    \begin{enumerate}
        \item[(i)] every separation induced by $\widehat{\Tcal}$ is also induced by $\Tcal$;
        \item[(ii)] for every edge $tu \in E(T)$ that induces a separation not induced by $\widehat{\Tcal}$ we have $P_t = P_u$;
        \item[(iii)] for every edge $tu \in E(\widehat{T})$ precisely one of ${P}_{t}$ or ${P}_{u}$ is a proper subset of the other;
        \item[(iv)] if $\Tcal$ distinguishes two profiles $Q_1$ and $Q_2$ efficiently, then so does $\widehat{\Tcal}$;
        \item[(v)] if $\Tcal$ is canonical, then $\widehat{\Tcal}$ is canonical as well.
    \qed
    \end{enumerate}
\end{remark}

\begin{lemma}\label{can_choice}
    Let $K$ be a complete subgraph of $G$ and $\widehat{\Tcal}$ as in Construction~\ref{constr:T_hat}.
    Then there is a node $t$ of $\widehat{T}$ with $V(K) \subseteq {P}_t$ such that ${P}_t$ is fixed by every automorphism of $G$ fixing $K$.
\end{lemma}

\begin{proof}
    As $K$ is complete, there is a node $u \in V(\widehat{T})$ with $V(K) \subseteq {P}_{u}$.
    
    Let $W$ be the subforest of nodes $w$ with $K \subseteq {P}_{w}$, which is connected as~$\widehat \Tcal$ is a tree-decomposition.
    Now $W$ either has a central vertex $t$ or a central edge~$tu$ such that ${P}_u$ is a proper subset of ${P}_t$ (cf Remark~\ref{gluing_aux_td}\,(iii)).
    In both cases ${P}_t$ is fixed by the automorphisms of $G$ that fix $K$.
\end{proof}

\begin{construction}\label{constr:T_bar}
    Let $\Tcal = \big(T,(P_t)_{t \in V(T)}\big)$ be a tree-decomposition of $G$.
    For each $t \in V(T)$ let ${\Tcal^t = \big(T^t,(P_{u}^{t})_{u \in V(T_t)}\big)}$ be a tree-decomposition of the torso $H_t$. 
    For each $\Tcal^t$ let $\widehat{\Tcal}^t$ be as in Construction~\ref{constr:T_hat}.
    For $e = t u \in E(T)$ let $A_e$ denote the adhesion set $P_{t} \cap P_{u}$.    
    Since $H_t[A_e]$ is complete, we can apply Lemma~\ref{can_choice}: there is a node $\gamma(t,u)$ of $\widehat{T}^t$ with $A_e \subseteq {P}_{\gamma(t,u)}^t$ 
such that ${P}_{\gamma(t,u)}^t$ is fixed by every automorphism of $H_t$ fixing $K$. 
    
    We obtain a tree $\overline{T}$ from the disjoint union of the trees $\widehat{T}^t$ for all $t \in V(T)$ by adding the edges $\gamma(t,u)\gamma(u,t)$ for each $tu\in E(T)$.
    Let $\overline{P}_u$ be ${P}_u^{t}$ for the unique $t \in V(T)$ with $u \in V(\widehat{T}^{t})$.
    Then $\overline{\Tcal} \defeq \big(\overline{T},(\overline{P}_t)_{t \in V(\overline{T})} \big)$ is a tree-decomposition of $G$.
\end{construction}

Two torsos $H_t$ and $H_u$ of $\Tcal$ are \emph{similar}, if there is an automorphism of $G$ that induces an isomorphism between $H_t$ and $H_u$.
The family $\big(\Tcal^t\big)_{t \in V(T)}$ is \emph{canonical} if all the $\Tcal^t$ are canonical and for any two similar torsos $H_t$ and $H_u$ of $\Tcal$ every automorphism of $G$ that witnesses the similarity of $H_t$ and $H_u$ induces an isomorphism between $\Tcal^t$ and $\Tcal^u$.

\begin{lemma}\label{gluing_lemma}
    The tree-decomposition $\overline{\Tcal}$ as in Construction~\ref{constr:T_bar} satisfies the following:
    \begin{enumerate}
        \item[(i)] for $t \in V(T)$ every node $u \in V(T^t)$ is also a node of $\overline{T}$ and $\overline{P}_{u} = P_{u}^t$;
        \item[(ii)] every node $u \in V(\overline{T})$ that is not a node of any $T^t$ is a hub node;
        \item[(iii)] every separation $(A,B)$ induced by $\overline{\Tcal}$ is either induced by $\Tcal$ or there is a node $t \in V(T)$ such that $(A \cap P_t, B \cap P_t)$ is induced by~$\Tcal^t$;
        \item[(iv)] every separation induced by $\Tcal$ is also induced by~$\overline{\Tcal}$;
        \item[(v)] for every separation $(C,D)$ induced by $\widehat{\Tcal}^t$ there is a separation $(A,B)$ induced by $\overline{\Tcal}$ such that $A \cap B \subseteq P_t$ and $(A \cap P_t, B \cap P_t) = (C,D)$;
        \item[(vi)] if $\Tcal$ and the family of the $\Tcal^t$ are canonical, then $\overline{\Tcal}$ is canonical. 
    \end{enumerate}
\end{lemma}

\begin{proof}
    Whilst (i) is true by construction, the nodes added in the construction of $\widehat{\Tcal}^t$ are hub nodes by definition, yielding (ii).
    Furthermore, (iii), (iv) and (v) follow by construction with Remark~\ref{gluing_aux_td}\,(i) and the observation that for all ${tu \in E(T)}$ the adhesion sets ${\overline{P}_{\gamma(t,u)} \cap \overline{P}_{\gamma(u,t)}}$ and ${P_t \cap P_u}$ are equal.
    Finally, (iv) follows with Remark~\ref{gluing_aux_td}\,(v) and Lemma~\ref{can_choice} from the construction.
\end{proof}

\subsection{Obtaining tree-decompositions from almost nested sets of separations}

Theorem~\ref{tree_from_N} gives a way how to transform a nested set of separations into a tree-decomposition.
In this subsection, we extend this to sets of `almost nested' separations. 

For a separation $(A,B)$ of $G$ and $X \subseteq V(G)$, the pair $(A \cap X, B \cap X)$ is a separation of $G[X]$, which we call the \emph{restriction} $(A,B)\restricted X$ of $(A,B)$ to~$X$.
Note that $(A,B)\restricted X$ is proper if and only if $(A,B)$ separates $X$.
The \emph{restriction} $S\restricted X$ to $X$ of a set $S$ of separations of $G$ to $X$ consists of the proper separations $(A,B)\restricted X$ with $(A,B) \in S$.

For a set $S$ of separations of $G$ let $\min_{\mathrm{ord}}(S)$ denote the set of those separations in $S$ with minimal order.
Note that if $S$ is non-empty, then so is $\min_{\mathrm{ord}}(S)$, and that $\min_{\mathrm{ord}}$ commutes with graph isomorphisms.

A finite sequence $(\beta_0,\ldots,\beta_n)$ of vertex sets of $G$ is called an $S$-\emph{focusing sequence} if
\begin{itemize}
    \item[(F1)] $\beta_0 = V(G)$;
    \item[(F2)] for all $i < n$, the separation system $N_{\beta_i}$ generated by $\min_{\mathrm{ord}}(S\restricted \beta_i)$ is non-empty and is nested with the set $S\restricted \beta_i$;
    \item[(F3)] $\beta_{i+1}$ is an $N_{\beta_i}$-block of $G[\beta_i]$ .
\end{itemize}

An $S$-focusing sequence $(\beta_0,\ldots,\beta_n)$ is \emph{good} if
\begin{itemize}
    \item[(F$\ast$)] the separation system $N_{\beta_n}$ generated by $\min_{\mathrm{ord}}(S\restricted \beta_n)$ is nested with the set $S\restricted \beta_{n}$.
\end{itemize}

Note that for an $S$-focusing sequence $(\beta_0, \ldots, \beta_n)$ we obtain $\beta_n \subseteq \beta_{n-1} \subseteq \ldots \subseteq \beta_0$.
The set of all $S$-focusing sequences is partially ordered by extension, where $(V(G))$ is the smallest element.
The subset $\Fcal_S$ of all good $S$-focusing sequences is downwards closed in this partial order.

\begin{lemma}\label{focusing_claim1}
    Let $(\beta_0, \ldots, \beta_n) \in \Fcal_S$ and let $(A,B) \in S$. If $(A,B)\restricted\beta_{n}$ is proper, then $A \cap B \subseteq \beta_n$.
\end{lemma}

\begin{proof}
    By assumption $(A,B)\restricted\beta_n$ is proper, hence there are $a \in (\beta_n \cap A) \setminus B$ and $b \in (\beta_n \cap B) \setminus A$.
    Since $\beta_n \subseteq \beta_{i}$ for all $i \leq n$ the separations $(A,B)\restricted\beta_i$ are proper as well.
    Suppose for a contradiction there is a vertex $v \in (A \cap B) \sm \beta_n$.
    Let $j < n$ be maximal with  $v \in \beta_j$.
    Since $\beta_{j+1}$ is an $N_{\beta_j}$-block of $G[\beta_{j}]$, there is a separation $(C,D) \in N_{\beta_j}$ with $v \in C \sm D$ and $\{a,b\} \subseteq \beta_n \subseteq \beta_{j+1} \subseteq D$.

    Now $a$, $b$ and $v$ witness that $(A,B)\restricted\beta_j$ and $(C,D)$ are not nested:
    Indeed, $a$ witnesses that $D$ is not a subset of $B\cap \beta_j$. Similarly, $b$ witnesses that $D$ is not a subset of $A\cap \beta_j$. But $v$ witnesses that neither $A\cap \beta_j$ nor $B\cap \beta_j$ is a subset of $D$.
    Thus we get a contradiction to the assumption that $N_{\beta_j}$ is nested with the set $S\restricted\beta_j$.
\end{proof}

A set $S$ of separations of $G$ is \emph{almost nested} if all $S$-focusing sequences are good.
In this case the maximal elements of $\Fcal_S$ in the partial order are exactly the $S$-focusing sequences $(\beta_0, \ldots, \beta_n)$ with $N_{\beta_n} = \emptyset$, and hence $S\restricted\beta_n = \emptyset$.

\begin{lemma}\label{focusing_claim2}
    Let $S$ be an almost nested set of separations of $G$.
    \begin{itemize}
        \item[(i)] If $(\beta_0, \ldots, \beta_n) \in \Fcal_S$ is maximal, then $\beta_n$ is an $S$-block.
        \item[(ii)] If $b$ is an $S$-block, there is a maximal $(\beta_0, \ldots, \beta_n) \in \Fcal_S$ with $\beta_n = b$.
    \end{itemize}
\end{lemma}

\begin{proof}
    Let $(\beta_0, \ldots, \beta_n) \in \Fcal_S$ be maximal.
    Then $S\restricted\beta_n$ is empty, i.e.\ no $(A,B) \in S$ induces a proper separation of $G[\beta_n]$.
    Hence $\beta_n$ is $S$-inseparable.
    For every $v \in V(G) \sm \beta_n$ there is an $i < n$ and a separation in $N_{\beta_i}$ separating $v$ from $\beta_n$. Hence $\beta_n$ is an $S$-block.
    
    Conversely given an $S$-block $b$, let $(\beta_0, \ldots, \beta_n) \in \Fcal_S$ be maximal with the property $b \subseteq \beta_n$, which exists since $(V(G)) \in \Fcal_S$ and since $\Fcal_S$ is finite.
    Since $b$ is $N_{\beta_n}$-inseparable, there is some $N_{\beta_n}$-block $\beta_{n+1}$ containing $b$.
    The choice of $(\beta_0, \ldots, \beta_n)$ implies that $(\beta_0, \ldots, \beta_{n+1}) \notin \Fcal_S$ and hence $N_{\beta_n} = \emptyset$, i.e.\ $(\beta_0, \ldots, \beta_n)$ is a maximal element of $\Fcal_S$.
    Thus $\beta_n$ is an $S$-block with $b \subseteq \beta_n$ and hence $b = \beta_n$.
\end{proof}

\begin{construction}\label{main_construction}
    Let $S$ be an almost nested set of separations of $G$.
    We recursively construct for every $S$-focusing sequence $(\beta_0, \ldots, \beta_n)$ a tree-decomposition $\Tcal^{\beta_n}$ of $G[\beta_n]$ so that the tree-decomposition $\Tcal^{V(G)} \eqdef \Tcal(S)$ for the smallest $S$-focusing sequence $(V(G))$ is a tree-decomposition of~$G$.
    
    For each maximal $S$-focusing sequence $(\beta_0, \ldots, \beta_m)$ we take the trivial tree-decomposition of $G[\beta_m]$ with only a single part.
    Suppose that $\Tcal^\beta$ has already been defined for every successor $(\beta_0, \ldots, \beta_n, \beta)$ of $(\beta_0, \ldots, \beta_n)$.
    To define $\Tcal^{\beta_n}$ we start with the tree-decomposition $\Tcal(N_{\beta_n})$ of $G[\beta_n]$ as given by Theorem~\ref{tree_from_N}.
    For each hub node $h$ we take the trivial tree-decomposition of $H_h$ and for each node $t$ whose part is an $N_{\beta_n}$-block $\beta$, we take $\Tcal^{\beta}$ given from the $S$-focusing sequence $(\beta_0, \ldots, \beta_n, \beta)$.
    This is indeed a tree-decomposition of the torso $H_t$, which we will show in Theorem~\ref{tree_from_almost_nested}.
    Hence we can apply Construction~\ref{constr:T_bar} to $\Tcal(N_{\beta_n})$ and the family of tree-decompositions of the torsos to get $\Tcal^{\beta_n}$.
\end{construction}

Given an $S$-focusing sequence $(\beta_0, \ldots, \beta_n)$, any two separations in $N_{\beta_n}$ have the same order $\ell$.
We call $\ell$ the \emph{rank} of $(\beta_0, \ldots, \beta_n)$.
If $N_{\beta_n}$ is empty, we set the rank to be $\infty$.

For an almost nested set $S$ of separations of $G$ two $S$-focusing sequences $(\beta_0, \ldots, \beta_n)$ and $(\alpha_0, \ldots \alpha_m)$ are \emph{similar} if there is an automorphism $\psi$ of~$G$ inducing an isomorphism between $G[\beta_n]$ and $G[\alpha_m]$.
Similar $S$-focusing sequences clearly have the same rank.
If $S$ is canonical, then $\psi$ induces an isomorphism between $\Tcal(N_{\beta_n})$ and $\Tcal(N_{\alpha_m})$ as obtained from Theorem~\ref{tree_from_N}.

\begin{theorem}\label{tree_from_almost_nested}
    The tree-decomposition $\Tcal(S)$ as in Construction~\ref{main_construction} is well-defined and satisfies the following:
    \begin{enumerate}
	    \item[(i)] every $S$-block of $G$ is a part of $\Tcal(S)$;
	    \item[(ii)] every part of $\Tcal(S)$ is either an $S$-block of $G$ or a hub;
	    \item[(iii)] for every separation $(A,B)$ induced by $\Tcal(S)$ there is a separation $(A',B') \in S$ such that $A \cap B = A' \cap B'$;
	    \item[(iv)] if $S$ is canonical, then so is $\Tcal(S)$.
    \end{enumerate}
\end{theorem}

\begin{proof}
    We show inductively that for any $S$-focusing sequence $(\beta_0, \ldots, \beta_n)$ the tree-decomposition $\Tcal^{\beta_n}$ has the following properties:
    \begin{itemize}
        \item[(a)] every $S$-block included in $\beta_n$ is a part of $\Tcal^{\beta_n}$;
        \item[(b)] every part of $\Tcal^{\beta_n}$ is either an $S$-block or a hub;
        \item[(c)] every separation $(A,B)$ induced by $\Tcal^{\beta_n}$ is proper;
        \item[(d)] and for every separation $(A,B)$ induced by $\Tcal^{\beta_n}$ there is a separation ${(A',B') \in S}$ and an $S$-focusing sequence $(\beta_0, \ldots, \beta) \geq (\beta_0, \ldots, \beta_n)$ 
            such that $(A',B')\restricted\beta = (A,B)$.
    \end{itemize}
    Furthermore we show for canonical $S$ by induction, that 
    \begin{itemize}
        \item[(e)] if $(\alpha_0, \ldots, \alpha_m)$ and $(\beta_0, \ldots, \beta_n)$ are similar, then $\Tcal^{\alpha_m}$ and $\Tcal^{\beta_n}$ are isomorphic;
        \item[(f)] $\Tcal^{\beta_n}$ is canonical.
    \end{itemize}
    
    The tree-decompositions for the maximal $S$-focusing sequences satisfy (a) and (b) by Lemma~\ref{focusing_claim2}, and (c) and (d) since their trees do not have any edges.
    If for two $S$-blocks $b_1$ and $b_2$ there is an isomorphism between $G[b_1]$ and $G[b_2]$ induced by an automorphism of $G$, then clearly the tree-decompositions are isomorphic.
    Hence (e) and (f) hold for all $S$-focusing sequences of rank $\infty$.
    
    Suppose for our induction hypothesis that for every $S$-focusing sequence $(\alpha_0, \ldots, \alpha_m)$ with rank greater than $r$ the tree-decomposition $\Tcal^{\alpha_m}$ of $G[\alpha_m]$ has the desired properties.
    Let $(\beta_0, \ldots, \beta_n)$ be an $S$-focusing sequence of rank $r$.
    Then for each successor $(\beta_0, \ldots, \beta_n, \beta)$ the tree-decomposition $\Tcal^\beta$ is indeed a tree-decomposition of the corresponding torso:
    for a separation $(A,B)$ induced by $\Tcal^\beta$ consider $(A',B')$ as given in (d).
    By (F$\ast$) we obtain that $(A',B')\restricted\beta_n = (A,B)$ is nested with $N_{\beta_n}$, hence $(A,B)$ does not separate any adhesion set in $H_t$.
    Hence $\Tcal^{\beta_n}$ is indeed well-defined.
    
    Lemma~\ref{gluing_lemma}\,(i), (ii) and (iii) and the induction hypothesis yield (a), (b) and (c) for $\Tcal^{\beta_n}$.
    Also by Lemma~\ref{gluing_lemma}\,(iii) for a separation $(A,B)$ induced by $\Tcal^{\beta_n}$ either ${(A,B) \in N_{\beta_n} \subseteq S\restricted\beta_n}$ or $(A,B)$ induces a separation in $\Tcal^\beta$ for an $N_{\beta_n}$-block $\beta$ on the corresponding torso.
    In the first case $(\beta_0, \ldots, \beta_n)$ is the desired $S$-focusing sequence for (d) and in the second case the induction hypothesis yields $(A',B') \in S$ and the desired $S$-focusing sequence extending $(\beta_0, \ldots, \beta_n, \beta)$.
    Hence (d) holds for $\Tcal^{\beta_n}$.
    
    Suppose $S$ is canonical.
    Let $(\alpha_0, \ldots, \alpha_m)$ be similar to $(\beta_0, \ldots, \beta_n)$.
    Then every automorphism of $G$ that witnesses the similarity also witnesses that $\Tcal(N_{\alpha_m})$ and $\Tcal(N_{\beta_n})$ are isomorphic.
    Hence any torso of $\Tcal(N_{\alpha_m})$ is similar to the corresponding torso of $\Tcal(N_{\beta_n})$ and by induction hypothesis the tree-decompositions of the torsos are isomorphic. 
    Therefore following Construction~\ref{constr:T_bar} yields (e).
    If two torsos $H_t$ and $H_u$ of $\Tcal(N_{\beta_n})$ are similar, then either $V(H_t)$ and $V(H_u)$ are $N(\beta_n)$-blocks whose corresponding $S$-focusing sequences are similar and have rank greater than $r$, or they are hubs.
    If they are $N_{\beta_n}$-blocks, the chosen tree-decompositions are isomorphic by the induction hypothesis.
    If they are hubs, the chosen trivial tree-decompositions are isomorphic as witnessed by every automorphism of $G$ witnessing the similarity of $H_t$ and $H_u$.
    Hence this family of tree-decompositions of the torsos of $\Tcal(N_{\beta_n})$ is canonical and with Lemma~\ref{gluing_lemma}\,(vi) we get (f).
        
    Inductively the tree-decomposition $\Tcal^{V(G)} = \Tcal(S)$ of $G$ satisfies (i), (ii) and (iv) by (a), (b) and (f).
    Finally, (iii) follows from (c), (d) and Lemma~\ref{focusing_claim1}.
\end{proof}

\subsection{Extending a nested set of separations}

In this subsection we give a condition for when we can extend a nested set of separations so that it distinguishes any two distinguishable profiles in a given set $\Pcal$ efficiently.

Let $N$ be a nested separation system of $G$ and $\Tcal(N) = \big(T,(P_t)_{t \in V(T)}\big)$ be the tree-decomposition of~$G$ as in Theorem~\ref{tree_from_N}.
Recall that a separation $(A,B)$ of $G$ nested with $N$ \emph{induces} a separation $(A \cap P_t, B \cap P_t)$ of each torso $H_t$. 
An {$\ell$-profile} $\widetilde{Q}$ of $H_t$ is \emph{induced} by a $k$-profile $Q$ of $G$ if for every ${(A',B') \in \widetilde{Q}}$ there is an $(A,B) \in Q$ which induces $(A',B')$ on $H_t$.

\begin{construction}\label{constr:profile_on_torso}
    Let $t \in V(T)$ and let $Q$ be a $k$-profile of $G$.
    We construct a profile $\widetilde{Q}^t$ of the torso $H_t$ which is induced by $Q$.
    \paragraph{ \bf Case 1:} $Q$ inhabits $P_t$.\\
    Let $(A,B)$ be a proper separation of $H_t$ of order less than $k$. 
    By Lemma~\ref{profiles_are_havens}, there is a unique component $C$ of $G - (A\cap B)$ with ${(V(G)\sm C,C\cup N(C))\in Q}$. 
    As~$Q$ is consistent and inhabits $P_t$, the set $C \cap P_t$ is non-empty and either a subset of $A \sm B$ or $B \sm A$, but not both.
    If $(C \cap P_t) \subseteq (B \sm A)$, then we let $(A,B)\in \widetilde{Q}^t$. Otherwise we let $(B,A)\in \widetilde{Q}^t$.
    
    \paragraph{ \bf Case 2:} $Q$ does not inhabit $P_t$ and ${(V(G) \sm C, C \cup N(C)) \notin Q}$ for all components $C$ of $G - P_t$.\\
    Let $(A,B)$ be a proper separation of $H_t$ of order less than $k$. 
    By Lemma~\ref{profiles_are_havens}, there is a unique component $C$ of $G - (A\cap B)$ with ${(V(G)\sm C,C\cup N(C))\in Q}$.
    Since $C$ is not a component of $G - P_t$, the set $C \cap P_t$ is non-empty by assumption, 
    and we define $\widetilde{Q}^t$ as above.
    \paragraph{ \bf Case 3:} $Q$ does not inhabit $P_t$ and there is a component $C$ of $G - P_t$ such that ${(V(G) \sm C, C \cup N(C)) \in Q}$.\\
    Let $m$ denote the size of the neighbourhood of $C$.
    Let $b$ be the $m$-block of~$H_t$ containing~$N(C)$.
    For $\widetilde{Q}^t$ we take the $m$-profile induced by $b$.
\end{construction}

The following is straightforward to check:
\begin{remark}\label{profile_on_torso1}
    The set $\widetilde{Q}^t$ as in Construction~\ref{constr:profile_on_torso} is a profile of~$H_t$ induced by~$Q$.
    Moreover, if $Q$ is $r$-robust, then so is $\widetilde{Q}^t$.
    \qed
\end{remark}

The next remark is a direct consequence of the relevant definitions.

\begin{remark} \label{profile_on_torso}
    Let $Q_1$ and $Q_2$ be profiles of $G$.
    \begin{enumerate}
        \item[(i)] If a separation $(A,B)$ of $G$ nested with $N$ distinguishes $Q_1$ and $Q_2$ efficiently, 
            then the induced separation ${(A \cap P_t, B \cap P_t)}$ of $H_t$ distinguishes $\widetilde{Q}_1^t$ and~$\widetilde{Q}_2^t$ efficiently for any part $P_t$ where it is proper;
        \item[(ii)] if a separation $(A,B)$ of some torso~$H_t$ distinguishes $\widetilde{Q}_1^t$ and~$\widetilde{Q}_2^t$, then any separation of $G$ that induces $(A,B)$ on $H_t$ distinguishes $Q_1$ and $Q_2$.
        \qed
    \end{enumerate}
\end{remark}

\begin{lemma}\label{refine_claim1}
    Let $Q_1$ and $Q_2$ be profiles of $G$ which are not already distinguished efficiently by $N$.
    Let $(A,B)$ distinguish them efficiently such that it is nested with $N$.
    Then there is a part $P_t$ of $\Tcal(N)$ such that the induced separation $(A \cap P_t, B \cap P_t)$ of the torso $H_t$ is proper.
\end{lemma}

\begin{proof}
    Since $(A,B)$ is nested with $N$, there is a part $P_t$ such that ${A \cap B \subseteq P_t}$.
    Suppose that $(A \cap P_t, B \cap P_t)$ is not proper.
    Without loss of generality let $(B \sm A) \cap P_t$ be empty and let ${(A,B) \in Q_1}$.
    
    By Lemma~\ref{profiles_are_havens} we obtain a component $K$ of $G - (A \cap B)$ such that ${(A,B) \leq (V(G) \sm K, K \cup N(K)) \in Q_1}$.
    By consistency of $Q_2$ the separation ${(V(G) \sm K, K \cup N(K))}$ still distinguishes $Q_1$ and $Q_2$, and since $(A,B)$ distinguishes $Q_1$ and $Q_2$ efficiently, the neighbourhood of $K$ is ${A \cap B}$.
    Let $u$ be the neighbour of $t$ such that the by $t u$ induced separation $(C_t, D_t) \in N$ satisfies ${K \cup N(K) \subseteq D_t}$.
    If $(B \sm A) \cap P_u$ is empty, we obtain ${(C_u, D_u) \in Q_1}$ as before and by construction we obtain ${(C_t, D_t) < (C_u, D_u)}$.
    
    Among all parts $P_t$ containing $A \cap B$ such that $(B \sm A) \cap P_t$ is empty, we choose a part $P_x$ such that $(C_x, D_x)$ is maximal with respect to the ordering of separations.
    Let $y$ denote the neighbour of $x$ such that $xy$ induces $(C_x, D_x)$.
    There is a vertex ${v \in (C_x \cap D_x) \sm (A \cap B)}$, since otherwise $(C_x, D_x)$ would distinguish $Q_1$ and $Q_2$ efficiently.
    Since we assumed that $(B \sm A) \cap P_x$ is empty, we deduce that $v \in A \sm B$.
    Therefore $(A \sm B) \cap P_y$ is not empty.
    Hence if $(A \cap P_y, B \cap P_y)$ on $H_y$ were improper, then ${(B \sm A) \cap P_y}$ would be empty and $(C_y, D_y)$ would contradict the maximality of $(C_x,D_x)$.
\end{proof}

For a nested separation system $N$ let $S_{<k}^N$ be the set of separations of order less than $k$ of $G$ nested with $N$.

\begin{construction}\label{constr:refine_N} 
    Let $N \subseteq S_{<{r+1}}$ be a nested separation system of~$G$ and let $\Pcal$ be a set $r$-robust $\ell$-profiles of~$G$ for some values $\ell \leq r+1$,
    such that $S_{< r+1}^N$ distinguishes any two distinguishable profiles in $\Pcal$ efficiently.
    
    Let $\Tcal(N) = \big( T, (P_t)_{t \in V(T)} \big)$ be as in Theorem~\ref{tree_from_N} 
    and let $\Pcal^t$ be the set of profiles $\widetilde{Q}^t$ of $H_t$ for $Q \in \Pcal$. 
    Applying Theorem~\ref{canon_td_profiles} to the graphs $H_t$ and the maximal $k$ of any $k$-profile in $\Pcal^t$, we get a tree-decomposition $\Tcal^t$ of~$H_t$ that distinguishes every two distinguishable profiles in $\Pcal^t$ efficiently. 
    Note that if $\Pcal$ is canonical, then the family $(\Tcal^t)_{t\in V(T)}$ is canonical as well.
    By applying Lemma~\ref{gluing_lemma} we obtain a tree-decomposition~$\overline{\Tcal}$ and the corresponding nested system $\overline{N}$ of separations of order at most~$r$ induced by~$\overline{\Tcal}$.
\end{construction}

\begin{theorem}\label{refine_N}
    The nested separation system $\overline{N}$ as in Construction~\ref{constr:refine_N} satisfies the following.
    \begin{itemize}
        \item[(i)] $N \subseteq \overline{N}$;
        \item[(ii)] $\overline{N}$ distinguishes every two distinguishable profiles in $\Pcal$ efficiently;
        \item[(iii)] if $N$ and $\Pcal$ are canonical, then so is $\overline{N}$.
    \end{itemize}
\end{theorem}

\begin{proof}
    Lemma~\ref{gluing_lemma}\,(iv) yields (i).
    For (ii), consider two distinguishable profiles ${Q_1, Q_2 \in \Pcal}$ not already distinguished efficiently by $N$.
    By assumption, there is some ${(A,B) \in S_{< r+1}^N}$ distinguishing $Q_1$ and $Q_2$ efficiently.
    
    By Lemma~\ref{refine_claim1} and Remark~\ref{profile_on_torso}\,(i) there is a part $P_t$ of $\Tcal(N)$ such that $\widetilde{Q}_1^t$ and $\widetilde{Q}_2^t$ are distinguished efficiently by ${(A \cap P_t, B \cap P_t)}$.
    Hence Theorem~\ref{canon_td_profiles}, Remark~\ref{gluing_aux_td}\,(iv), Lemma~\ref{gluing_lemma}\,(v) and Remark~\ref{profile_on_torso}\,(ii) yield a separation of order $\lvert A \cap B \rvert$ in~$\overline{N}$ distinguishing $Q_1$ and $Q_2$, yielding (ii).
    
    Finally, (iii) holds by construction.
\end{proof}

\section{Proof of the main result}\label{section:proof}

Given a $k$-block~$b$ and a component $C$ of $G-b$, then $(C\cup N(C), V(G)\sm C)$ is a separation.
By $S_k(b)$ we denote the set of all those separations. 
Note that $S_k(b)$ is a nested set of separations, while for different ($r$-robust) $k$-blocks $b$, $b'$ the union $S_k(b) \cup S_k(b')$ need not to be nested \cite{canon2}.

\begin{lemma}\label{separability}
    Let $b$ be a $k$-block of $G$.
Then $b$ is separable if and only if every separation in $S_k(b)$ has order less than $k$.
\end{lemma}

\begin{proof}
    For the `only if'-implication, let $\Tcal = \big(T, (P_t)_{t \in V(T)}\big)$ be a tree-decom{-}position of adhesion less than $k$ of $G$ with $P_t = b$ for some $t \in V(T)$. 
    Let $C$ be a component of $G-b$.
    There is a separation $(A,B)$ induced by $\Tcal$ with $C\se A\sm B$ and $b\se B$. 
    Hence $N(C)\se A\cap B$, and so has less than $k$ vertices. 

    For the `if'-implication, just consider the star-decomposition induced by $S_k(b)$, whose central part is $b$.
    This tree-decomposition has adhesion less than $k$ if and only if all separations in $S_k(b)$ have order less than $k$.
\end{proof}

\begin{remark}\label{inseparable_separator}
    Let $b$ be a $k$-block of $G$.
    For all $(A,B) \in S_k(b)$ the separator $A \cap B$ is a subset of $b$.
    \qed 
\end{remark}


Given some $r \in \mathbb{N}$ and a set $\Bcal$ of distinguishable\footnote{A set of blocks is \emph{distinguishable} if the set of induced profiles is distinguishable.} $r$-robust $k$-blocks for some values $k \leq r+1$, we define
\[
    S(\Bcal) \defeq \bigcup \big\{ S_k(b) \cap S_{<k}\ \big|\ b \text{ is a } k \text{-block in } \Bcal \big\}.
\]
Note that if the set of profiles induced by $\Bcal$ is canonical, then so is $S(\Bcal)$.

\begin{lemma}\label{S_blocks}
    Every separable $k$-block $b \in \Bcal$ is an $S(\Bcal)$-block. 
\end{lemma}

\begin{proof}
    Suppose for a contradiction there is a $k'$-block $b' \in \Bcal$ and a separation ${(A,B) \in S_{k'}(b') \cap S_{<k'} \subseteq S(\Bcal)}$ separating $b$.
    Consider a separation $(C,D)$ distinguishing $b$ and $b'$ efficiently with $b \subseteq C$ and $b' \subseteq D$.
    Since ${\lvert C \cap D \rvert < k}$, there is a vertex ${v \in b \sm (C \cap D)}$.
    And since $(A \cap B) \subseteq b' \subseteq D$, the link $\ell_C$ is empty.
    Therefore we deduce that either $v \in A \sm B$ or $v \in B \sm A$.
    Let $w$ denote a vertex of $b$ such that $(A,B)$ separates $v$ and $w$.
    Both the corner separations $(A \cap C, B \cup D)$ and $(B \cap C, A \cup D)$ have order at most $\lvert C \cap D \rvert < k$.
    But one of them separates $v$ from $w$, contradicting the $(< k)$-inseparability of $b$.
    Hence $b$ is $S(\Bcal)$-inseparable.
    
    Let $X$ be an $S(\Bcal)$-inseparable set including $b$ and let $v \in V(G) \sm b$. 
    Then there is some $(A,B) \in S_k(b)$ separating $b$ from $v$.
    Lemma~\ref{separability} implies that ${(A,B) \in S_k(b) \cap S_{<k} \subseteq S(\Bcal)}$ and thus $v$ is not in~$X$.
    Hence $X=b$. 
\end{proof}

\begin{lemma}\label{almost_claim1}
    Let $(A,B)$ and $(C,D)$ be tight separations of $G$ such that $A\sm B$ is connected and the link $\ell_A$ is empty. Then $(A,B)$ and $(C,D)$ are nested.
\end{lemma}

\begin{proof}
    Since $A \sm B$ is connected, either $\interior(A,C)$ or $\interior(A,D)$ is empty, say $\interior(A,C)$. 
    Thus there cannot be a vertex in the link $\ell_C$ because it would have a neighbour in $A\sm B$, which is impossible.
    Hence $(A,B)$ and $(C,D)$ are nested by Remark~\ref{corner_lemma}.
\end{proof}

\begin{lemma}\label{almost_claim2}
    Let $(A,B),(C,D)\in S(\Bcal)$ be crossing.
    Then the links $\ell_B$ and $\ell_D$ are empty.
    
    Moreover, the separation ${(K\cup N(K), V(G)\sm K)}$ for every component $K$ of $G[\interior(B, D)]$ is in $S(\Bcal)$ and its order is strictly less than the orders of both $(A,B)$ and $(C,D)$.
\end{lemma}
    
\begin{proof}
    Let $b_1$ and $b_2$ be blocks in $\Bcal$ such that ${(A,B) \in S_{k_1}(b_1) \cap S_{< k_1}}$ and ${(C,D) \in S_{k_2}(b_2) \cap S_{< k_2}}$.
    We may assume that the order $k_2$ of $b_2$ is at most the order $k_1$ of $b_1$.
    By Lemma~\ref{almost_claim1}, there are vertices $v_A\in \ell_A$ and $v_C\in \ell_C$.
    By Remark~\ref{inseparable_separator}, $v_C\in b_1$. As  $(C,D)$ cannot separate $b_1$, 
    the block $b_1$ is contained in $B\cap C$.
    In particular, the link $\ell_D$ is empty. 
    
    Let $X$ be a component of $G-C\cap D$ that contains a vertex $w$ of $b_2$. Note that $X$ is unique as $b_2$ is a $k_2$-block.
    As $\ell_D$ is empty, $X$ must be contained in $D\cap A$ or $D\cap B$.
    Since $b_2$ contains $v_A$, it must be contained in $D\cap A$.
    Indeed, otherwise the corner separation of $B\cap D$ would separated $w$ from~$v_A$.
    Hence $\ell_B$ is empty. 
    
    Let $K$ be an arbitrary component of ${G[\interior(B,D)]}$.
    Let ${E \defeq K \cup N(K)}$ and ${F \defeq V(G) \sm K}$.
    Since the center $c$ is a subset of ${b_1 \cap b_2}$ and since ${K \cap (b_1 \cup b_2)}$ is empty, $K$ is a component of both $G - b_1$ and $G-b_2$.
    Hence $(E,F)$ is in both $S_{k_1}(b_1)$ and $S_{k_2}(b_2)$. 
    And since ${E \cap F \subseteq c}$ and since $\ell_A$ and $\ell_C$ are not empty, we deduce that ${\abs{E \cap F} < \min\{\abs{A \cap B},\abs{C \cap D}\}}$.
\end{proof}

\begin{lemma}\label{S_almost_nested}
    $S(\Bcal)$ is almost nested.
\end{lemma}

\begin{proof}
    We have to show that every $S(\Bcal)$-focusing sequence $(\beta_0, \ldots, \beta_n)$ is good, i.e.\ $N_{\beta_n}$ is nested with $S(\Bcal) \restricted \beta_n$.
    Let $(\beta_0, \ldots, \beta_n)$ be an $S(\Bcal)$-focusing sequence.
    Let ${(A,B)\restricted\beta_n \in N_{\beta_n}}$ and ${(C,D)\restricted\beta_n \in S(\Bcal)\restricted\beta_n}$.
    If $(A,B)$ and $(C,D)$ are nested, then so are $(A,B)\restricted\beta_n$ and $(C,D)\restricted\beta_n$.
    Suppose $(A,B)$ and $(C,D)$ are crossing.
    By Lemma~\ref{almost_claim2} $\ell_B$ and $\ell_D$ are empty.
    If $\interior(B,D) \cap \beta_n$ is empty, then by Remark~\ref{corner_lemma} $(A,B)\restricted\beta_n$ and $(C,D)\restricted\beta_n$ are nested.
    Hence by Lemma~\ref{almost_claim2} it suffices to show that $(E \sm F) \cap \beta_n$ is empty for every ${(E,F) \in S(\Bcal)}$ with $E \subseteq B \cap D$ whose order is strictly smaller than the order of $(A,B)$.
    
    Since $(A,B)\restricted\beta_n$ is proper, there is a ${v \in \beta_n \sm B \subseteq \beta_n \sm E \subseteq (F \sm E) \cap \beta_n}$.
    Since $(A,B)\restricted\beta_n$ has minimal order among all separations in $S(\Bcal)\restricted\beta_n$, we deduce that $(E,F)\restricted\beta_n$ is improper and hence either $(F \sm E) \cap \beta_n$ or $(E \sm F) \cap \beta_n$ is empty.
    Now $v$ witnesses that $(E \sm F) \cap \beta_n$ is empty, as desired.
\end{proof}

\begin{lemma}\label{N_is_refinable} 
    Given $r \in \mathbb{N}$, let $\Pcal$ be a set of $r$-robust distinguishable $k$-profiles for some values ${k \leq r+1}$. 
    Let $N$ be a nested separation system such that for every $(C,D)\in N$, there is some $\ell$-profile in $\Pcal$ induced by an $\ell$-block $b$ with $(C \cap D) \se b$. 
    Then any two distinct ${P, Q \in \Pcal}$ are distinguished efficiently by a separation nested with $N$. 
\end{lemma}

\begin{proof}
    Let $(A,B)$ distinguish $P, Q \in \Pcal$ efficiently such that the number of separations in $N$ nested with $(A,B)$ is maximal.
    Without loss of generality let $(A,B) \in P$.
    Let ${k \defeq \lvert A \cap B \rvert}$.
    We prove that $(A,B)$ is nested with $N$. 

    Suppose for a contradiction that there is some $(C,D) \in N$ not nested with $(A,B)$.
    Let $b$ be an $(\ell+1)$-block such that ${(C \cap D) \se b}$ whose induced profile $P_{\ell+1}(b)$ is in $\Pcal$.
    \paragraph{ \bf Case 1: $k\leq \ell$. }
    Remark~\ref{inseparable_separator} implies that $C \cap D$ is $(\leq \ell)$-inseparable and hence one of the links $\ell_A$ or $\ell_B$ is empty.
    Without loss of generality let $\ell_B$ be empty.
    The orders of the corner separations $(A \cup D, B \cap C)$ and $(A \cup C, B \cap D)$ are less or equal than $\abs{A \cap B}$. Hence they are oriented by $P$ and $Q$.
    Applying Lemma~\ref{profiles_are_havens} to ${X \defeq A \cap B}$ and $P$ yields a component $K$ of $G-X$ with ${(V(G)\sm K, K \cup N(K)) \in P}$.
    In particular we get ${K \subseteq B \sm A}$ by consistency.
    Since $\ell_B$ is empty and $K$ is connected, we obtain ${K \subseteq C \sm D}$ or ${K \subseteq D \sm C}$.
    Therefore either ${(A \cup D, B \cap C)}$ or ${(A \cup C, B \cap D)}$ is in $P$ by consistency to ${(V(G)\sm K, K \cup N(K))}$, and not in $Q$ by consistency to $(B,A)$.
    
    Hence there is a corner separation of $(A,B)$ and $(C,D)$ distinguishing $P$ and $Q$ efficiently.
    By Lemma~\ref{fish_lemma} it is nested with every separation in $N$ that is also nested with $(A,B)$, as well as with $(C,D)$.
    Hence it crosses strictly less separations of $N$ than $(A,B)$, contradicting the choice of $(A,B)$.
    Thus $(A,B)$ is nested with $N$.

    \paragraph{ \bf Case 2: $k \geq \ell$. }
    We prove this case by induction on $k$ with Case~1 as the base case.
    By the efficiency of $(A,B)$, the separation $(C,D)$ does not distinguish $P$ and $Q$.
    Thus we may assume that $(C,D)$ is in both $P$ and $Q$. 
    If one of the corner separations $(A \cap D, B \cup C)$ or $(B \cap D, A \cup C)$ had order at most $k$, then it would violate the maximality of $(A,B)$ by Lemma~\ref{fish_lemma}.
    Indeed, it would be nested with every separation in $N$ that is also nested with $(A,B)$, as well as with $(C,D)$.
    
    Hence we may assume that both these corner separations have order larger than $k$ and therefore both links $\ell_A$ and $\ell_B$ are not empty. By Remark~\ref{submodularity}, the opposite corner separations $(A \cap C, B \cup D)$ and $(B \cap C, A \cup D)$
    have order strictly less than $\lvert C \cap D \rvert$ and are in $P_{\ell+1}(b)$ since $C \cap D \se b$.
    As $b$ is $r$-robust, $(C,D) \in P_{\ell+1}(b)$. 
    Hence $(C,D)$ distinguishes $P$ and $P_{\ell+1}(b)$.
    
    By the induction hypothesis, there is a separation $(E,F)$ of order at most~$\ell$ distinguishing $P$ and $P_{\ell+1}(b)$ efficiently that is nested with $N$. 
    We may assume that ${(E,F) \in P_{\ell+1}(b)}$ and  ${(F,E) \in P}$.
    Furthermore, $(E,F)$ does not distinguish $P$ and $Q$, since $\lvert E \cap F \rvert < \lvert A \cap B \rvert$.
    We claim that ${(C,D) \leq (F,E)}$. 
    Indeed, since $(C,D)$ and $(F,E)$ are nested and $P$ contains both of them, either ${(C,D) \leq (F,E)}$ or ${(F,E) \leq (C,D)}$.
    By consistency of $P_{\ell+1}(b)$, we can conclude that ${(C,D) \leq (F,E)}$.
    
    If the order of $(E \cap B, F \cup A)$ is at most $k$, then it would distinguish $P$ and~$Q$ efficiently.
    It would violate the maximality of $(A,B)$ by Lemma~\ref{fish_lemma} since it is nested with every separation in $N$ that is also nested with $(A,B)$, as well as with $(C,D)$ itself as $(C,D) \geq (E,F) \geq (E \cap B, F \cup A)$.
    Thus we may assume that  $(E \cap B, F \cup A)$ has order larger than $k$.
    Similarly we may assume that $(E \cap A, F \cup B)$ has order larger than $k$.

    Again by Remark~\ref{submodularity}, the opposite corner separations ${(F\cap A, E\cup B)}$ and ${(F\cap B, E\cup A)}$ have order less than $\lvert E\cap F \rvert$.
    But by construction they separate $\ell_A$ and $\ell_B$ and hence $b$, contradicting the fact that $b$ is $(\leq \ell)$-inseparable.
\end{proof}

\begin{theorem}\label{main_thm} 
    Let $G$ be a finite graph,  $r\in \mathbb{N}$ and let $\Pcal$ be a canonical set of $r$-robust distinguishable $\ell$-profiles for some values $\ell \leq r+1$. 

    Then $G$ has a canonical tree-decomposition $\Tcal$ that distinguishes efficiently every two distinct profiles in $\Pcal$, 
    and which has the further property that
    every separable block whose induced profile is in~$\Pcal$ is equal to the unique part of~$\Tcal$ in which it is contained.
\end{theorem}

\begin{proof} 
    Let $\Bcal$ be the set of blocks whose induced profiles are in $\Pcal$.
    We consider $S(\Bcal)$ as above.
    Lemma~\ref{S_almost_nested} and Construction~\ref{main_construction} yield a canonical tree-decomposition $\Tcal(S(\Bcal))$ where by Lemma~\ref{S_blocks} and Theorem~\ref{tree_from_almost_nested}\,(i) every separable $b \in \Bcal$ is equal to the unique part in which it is contained.
    
    Let $N$ be the nested separation system induced by $\Tcal(S(\Bcal)))$.
    With Lemma~\ref{N_is_refinable} we can apply Construction~\ref{constr:refine_N} to obtain $\overline{N}$, which by Theorem~\ref{refine_N}\,(ii) distinguishes the profiles in $\Pcal$ efficiently.
    
    It is left to show that no separation $(A,B) \in \overline{N} \sm N$ separates a separable $k$-block $b \in \Bcal$.
    Suppose for a contradiction that $(A,B) \in \overline{N} \sm N$ separates $b$.
    Let $P_t$ be the part of $\Tcal(S(\Bcal))$ with $P_t = b$.
    Note that since the adhesion sets $P_t \cap P_u$ for any edge $tu$ have size strictly smaller than $k$ and since the only profile in $\Pcal$ inhabiting $P_t$ is $P_k(b)$, no profile in $\Pcal$ induces an $\ell$-profile for some $\ell \geq k+1$ on the torso $H_t$.
    Then by Construction~\ref{constr:refine_N} and Lemma~\ref{gluing_lemma}\,(iii) the induced separation $(A \cap P_t, B \cap P_t)$ is a proper separation of $H_t$ distinguishing two $(\leq k)$-profiles of $H_t$ efficiently.
    But since $H_t$ has no proper $(< k)$-separation, it has no two distinguishable $(\leq k)$-profiles.
    
    Hence Theorem~\ref{tree_from_N} yields a tree-decomposition $\Tcal(\overline{N})$ with the desired properties.
\end{proof}

\begin{corollary}\label{cor_main}
    Every finite graph $G$ has a canonical tree-decomposition $\Tcal$ that distinguishes efficiently every two distinct maximal robust profiles, 
    and which has the further property that
    every separable block inducing a maximal robust profile is equal to the unique part of $\Tcal$ in which it is contained.
\end{corollary}

\begin{proof}
    Since the set of maximal robust profiles is by definition distinguishable, we can apply Theorem~\ref{main_thm}.
\end{proof}

\begin{corollary}\label{cor_main2}
    Every finite graph $G$ has a canonical tree-decomposition~$\Tcal$ of adhesion less than $k$ that distinguishes efficiently every two distinct $k$-profiles, 
    and which has the further property that
    every separable $k$-block is equal to the unique part of $\Tcal$ in which it is contained.
\end{corollary}

\begin{proof}
    By Remark~\ref{robust}\,(i) any $k$-profile is $(k-1)$-robust. Since the set of all $k$-profiles is by definition distinguishable, we can apply Theorem~\ref{main_thm}.
\end{proof}

Theorem~\ref{main_thm} fails if we do not require that $\Pcal$ is distinguishable:
\begin{example}\label{ex:distinguishable_necessary}
    Consider the graph obtained by two cliques $K_1$ and $K_2$ of size at least $k + 1 \geq 7$ sharing $k-1$ vertices, together with a vertex $v$ joined to two vertices of $K_1 - K_2$ and to two vertices of $K_2 - K_1$, see Figure~\ref{fig:ex4}.
    
    Then $K_1 \cup K_2$ is a separable $5$-block, as witnessed by the separation ${(\{v\} \cup N(v), K_1 \cup K_2)}$.
    But the two $(k+1)$-blocks $K_1$ and $K_2$ are only distinguished efficiently by ${(K_1 \cup \{v\}, K_2 \cup \{v\})}$.
    Since this separation crosses any separation separating $v$ from $K_1 \cup K_2$, there is no tree-decomposition that distinguishes $K_1$ and $K_2$ efficiently such that there is a part equal to $K_1 \cup K_2$.
    Moreover, even the union of the parts inhabited by $P_5(K_1 \cup K_2)$ in any tree-decomposition that distinguishes $K_1$ and $K_2$ efficiently contains with $v$ a vertex outside the block.
\end{example}

\begin{center}
    \begin{figure}[htpb]
        \begin{center}
            \includegraphics[height=3cm]{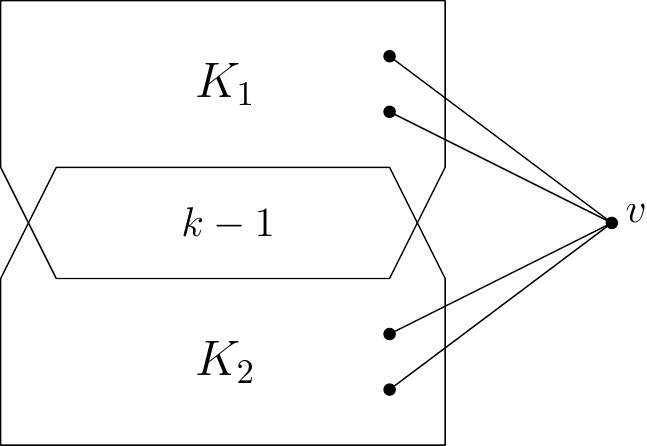}
            \caption{The graph of Example~\ref{ex:distinguishable_necessary}}
            \label{fig:ex4}
        \end{center}
    \end{figure}
\end{center}

\section*{Acknowledgement}\label{section:acknowledge}

We thank Matthias Hamann for proof reading.

\bibliographystyle{plain}
\bibliography{literature.bib}

\end{document}